\documentclass[12pt,oneside]{amsart}

\usepackage{amssymb,amsxtra,mathtools}

\usepackage[
  paper=a4paper,
  headsep=20pt,
  footskip=18pt,
  text={160mm,240mm},
  centering,
  includehead
]{geometry}

\usepackage{newtxtext,newtxmath}

\usepackage[bookmarks,bookmarksdepth=2,pdfencoding=auto,psdextra]{hyperref}
\usepackage{enumitem}
\setlist[enumerate,1]{label=(\arabic*),font=\upshape}
\usepackage[capitalize,noabbrev]{cleveref}

\crefname{enumi}{}{}
\crefname{equation}{}{}
\usepackage{thmtools}
\usepackage{graphicx,xcolor}
\usepackage{tikz}
\usetikzlibrary{cd,calc,arrows,decorations.pathreplacing}
\tikzcdset{arrow style=math font}
\usetikzlibrary{arrows,spath3,intersections,nfold}
\makeatletter
\tikzset{
  offset now/.code=
    \pgfgetpath\tikz@temp\pgfsetpath\pgfutil@empty\pgfoffsetpath\tikz@temp{#1}
}
\makeatother

\usepackage[mode=buildnew]{standalone}
\usepackage{comment}

\iftrue
  \makeatletter
  \def\@settitle{%
    \begin{flushleft}%
      \LARGE\bfseries
      \strut\@title\strut
    \end{flushleft}%
  }
  \def\@setauthors{%
    \begingroup
    \def\thanks{\protect\thanks@warning}%
    \trivlist
    \raggedright
    \large \@topsep30\p@\relax
    \advance\@topsep by -\baselineskip
    \item\relax
    \author@andify\authors
    \def\\{\protect\linebreak}%
    \authors
    \ifx\@empty\contribs
    \else
    ,\penalty-3 \space \@setcontribs
    \@closetoccontribs
    \fi
    \normalfont
    \endtrivlist
    \endgroup
  }
  \def\@setaddresses{\par
    \nobreak \begingroup
    \small\raggedright
    \def\author##1{\nobreak\addvspace\smallskipamount}%
    \def\\{\unskip, \ignorespaces}%
    \interlinepenalty\@M
    \def\address##1##2{\begingroup
      \par\addvspace\bigskipamount\noindent
      \@ifnotempty{##1}{(\ignorespaces##1\unskip) }%
      {\ignorespaces##2}\par\endgroup}%
    \def\curraddr##1##2{\begingroup
      \@ifnotempty{##2}{\nobreak\noindent\curraddrname
        \@ifnotempty{##1}{, \ignorespaces##1\unskip}\/:\space
        ##2\par}\endgroup}%
    \def\email##1##2{\begingroup
      \@ifnotempty{##2}{\nobreak\noindent E-mail address%
        \@ifnotempty{##1}{, \ignorespaces##1\unskip}\/:\space
        \ttfamily##2\par}\endgroup}%
    \def\urladdr##1##2{\begingroup
      \def~{\char`\~}%
      \@ifnotempty{##2}{\nobreak\noindent\urladdrname
        \@ifnotempty{##1}{, \ignorespaces##1\unskip}\/:\space
        \ttfamily##2\par}\endgroup}%
    \addresses
    \endgroup
    \global\let\addresses=\@empty
  }
  \def\@setabstracta{%
    \ifvoid\abstractbox
    \else
    \skip@30pt \advance\skip@-\lastskip
    \advance\skip@-\baselineskip \vskip\skip@
    \box\abstractbox
    \prevdepth\z@ 
    \vskip-20pt
    \fi
  }
  \renewenvironment{abstract}{%
    \ifx\maketitle\relax
    \ClassWarning{\@classname}{Abstract should precede
      \protect\maketitle\space in AMS document classes; reported}%
    \fi
    \global\setbox\abstractbox=\vtop \bgroup
    \normalfont\small
    \list{}{\labelwidth\z@
      \leftmargin0pc \rightmargin\leftmargin
      \listparindent\normalparindent \itemindent\z@
      \parsep\z@ \@plus\p@
      
    }%
  \item[\hskip\labelsep\bfseries\abstractname.]%
  }{%
    \endlist\egroup
    \ifx\@setabstract\relax \@setabstracta \fi
  }

  \def\ps@headings{\ps@empty
    \def\@evenhead{%
      \setTrue{runhead}%
      \normalfont\scriptsize
      \rlap{\thepage}\hfill
      \def\thanks{\protect\thanks@warning}%
      \leftmark{}{}}%
    \def\@oddhead{%
      \setTrue{runhead}%
      \normalfont\scriptsize
      \def\thanks{\protect\thanks@warning}%
      \rightmark{}{}\hfill \llap{\thepage}}%
    \let\@mkboth\markboth
  }\ps@headings

  \expandafter\def\csname section\endcsname{\@startsection{section}{1}%
    \z@{-1.4\linespacing\@plus-.5\linespacing}{.8\linespacing\@plus.3\linespacing}%
    {\normalfont\bfseries\Large}}
  \expandafter\def\csname subsection\endcsname{\@startsection{subsection}{2}%
    \z@{-.8\linespacing\@plus-.3\linespacing}{.5\linespacing\@plus.3\linespacing}%
    {\normalfont\bfseries\large}}
  \expandafter\def\csname subsubsection\endcsname{\@startsection{subsubsection}{3}%
    \z@{-.7\linespacing\@plus-.2\linespacing}{.3\linespacing\@plus.2\linespacing}%
    {\normalfont\bfseries\normalsize}}
  \expandafter\def\csname paragraph\endcsname{\@startsection{paragraph}{4}%
    \z@{.7\linespacing\@plus.2\linespacing}{-1.5ex}%
    {\normalfont\itshape}}
  \def\@secnumfont{\bfseries}

  \renewcommand\contentsnamefont{\bfseries}
  \def\@starttoc#1#2{\begingroup
    \setTrue{#1}%
    \par\removelastskip\vskip\z@skip
    \@startsection{}\@M\z@{\linespacing\@plus\linespacing}%
    {.5\linespacing}{
      \contentsnamefont}{#2}%
    \ifx\contentsname#2%
    \else \addcontentsline{toc}{section}{#2}\fi
    \makeatletter
    \@input{\jobname.#1}%
    \if@filesw
    \@xp\newwrite\csname tf@#1\endcsname
    \immediate\@xp\openout\csname tf@#1\endcsname \jobname.#1\relax
    \fi
    \global\@nobreakfalse \endgroup
    \addvspace{32\p@\@plus14\p@}%
    \let\tableofcontents\relax
  }
  \def\contentsname{Contents}
  \def\l@section{\@tocline{2}{.5ex}{0mm}{5pc}{}}
  \def\l@subsection{\@tocline{2}{0pt}{2em}{5pc}{}}
  \makeatother

\fi

\def\to{\mathchoice{\longrightarrow}{\rightarrow}{\rightarrow}{\rightarrow}}
\makeatletter
\newcommand{\shortxra}[2][]{\ext@arrow 0359\rightarrowfill@{#1}{#2}}
\def\longrightarrowfill@{\arrowfill@\relbar\relbar\longrightarrow}
\newcommand{\longxra}[2][]{\ext@arrow 0359\longrightarrowfill@{#1}{#2}}
\renewcommand{\xrightarrow}[2][]{\mathchoice{\longxra[#1]{#2}}%
  {\shortxra[#1]{#2}}{\shortxra[#1]{#2}}{\shortxra[#1]{#2}}}

\makeatother

\makeatletter
\def\addtagsub#1{\let\oldtf=\tagform@\def\tagform@##1{\oldtf{##1}\hbox{$_{#1}$}}}
\makeatother

\makeatletter
\def\Nopagebreak{\@nobreaktrue\nopagebreak}
\makeatother

\foreach \n in 
{Z,Q,R,C}
{\expandafter\xdef\csname\n\endcsname{\noexpand\mathbb{\n}}}

\foreach \n in 
{A,B,F,S}
{\expandafter\xdef\csname b\n\endcsname{\noexpand\mathbb{\n}}}

\foreach \n in 
{A,...,Z}
{\expandafter\xdef\csname c\n\endcsname{\noexpand\mathcal{\n}}}

\foreach \n in 
{dim,rk,rank,sign,sgn,Map,Ker,Coker,Im,Hom,End,GL,SL,O,SO,
 Tor,Ext,Gal,Irr,dis,Wh,Hopf,tr,Tr,vol,colim,grk,hofiber,
 lk,st,Mod,Aut,SL,Dax,dax,FQ,Homeo,Diff,Emb,Imm,fImm,ks,
 CAT,TOP,DIFF,PL}
{\expandafter\xdef\csname\n\endcsname{\noexpand\operatorname{\n}}}

\foreach \n/\s in 
{inte/int}
{\expandafter\xdef\csname\n\endcsname{\noexpand\operatorname{\s}}}

\foreach \n in 
{id,pr,PI,coll}
{\expandafter\xdef\csname\n\endcsname{\noexpand\mathrm{\n}}}

\foreach \n/\s in 
{csum/\#,bcsum/\natural/}
{\expandafter\xdef\csname\n\endcsname{\noexpand\mathbin{\s}}}

\def\rel{\ifmmode\;{\mathgroup\symoperators rel}\;\else rel\ \fi}
\def\sm{\smallsetminus}

\declaretheoremstyle[
  name=\ignorespaces,numbered=no,notefont=\bfseries,notebraces={}{},bodyfont=\itshape
]{theorem-entitled}
\declaretheoremstyle[
  name=\ignorespaces,numbered=no,notefont=\bfseries,notebraces={}{},bodyfont=\normalfont
]{definition-entitled}

\declaretheorem[name=Theorem,refname={Theorem,Theorems}]{theoremalpha}
\declaretheorem[name=Corollary,sibling=theoremalpha]{corollaryalpha}

\declaretheorem[numberwithin=section]{theorem}
\foreach \n in {
  question,corollary,proposition,lemma,conjecture,observation
}{\declaretheorem[sibling=theorem]{\n}}
\foreach \n in {
  definition,example,remark,notation,case
}{\declaretheorem[style=definition,sibling=theorem]{\n}}
\foreach \n in {
  assertion,claim
}{\declaretheorem[numbered=no]{\n}}

\declaretheorem[style=theorem-entitled]{theorem-named}
\foreach \n in {
  definition-named,conjecture-named,case-named
}{\declaretheorem[style=definition-entitled]{\n}}

\numberwithin{equation}{section}

\def\labelignored#1{\ignorespaces}

\begin{document}

\title
[Smoothing topological isotopy of surfaces in 4-manifolds]
{Smoothing topological isotopy of surfaces in 4-manifolds with a boundary geometric dual}

\author{Jae Choon Cha}
\address{
  Center for Research in Topology\\
  POSTECH\\
  Pohang Gyeongbuk 37673\\
  Republic of Korea
}
\email{jccha@postech.ac.kr}

\author{Byeorhi Kim}
\address{
  Center for Research in Topology\\
  POSTECH\\
  Pohang Gyeongbuk 37673\\
  Republic of Korea
}
\email{byeorhikim@postech.ac.kr}



\begin{abstract}
  We prove that topological isotopy implies smooth isotopy for neat surfaces with common nonempty boundary in a smooth orientable 4-manifold, in the presence of a geometric dual in the boundary.
  Together with our earlier work on surface smoothing, this shows that the natural map from smooth isotopy classes to topological isotopy classes of such surfaces with prescribed boundary is bijective.
\end{abstract}

\maketitle

\section{Introduction}

Smooth and topological classifications can differ dramatically in dimension four. What geometric conditions force them to agree? For embedded surfaces, such agreement means that every topological isotopy class admits a smooth representative, unique up to smooth isotopy.

Our main result establishes this agreement for surfaces in 4-manifolds under a boundary geometric dual hypothesis.
It applies to surfaces of arbitrary genus, including nonorientable surfaces and surfaces with several boundary components, without stabilization or additional hypotheses on the fundamental group of the 4-manifold.

In what follows, surfaces are compact and connected, embeddings are neat, and isotopies of embeddings are rel boundary, i.e., the boundary is fixed pointwise.
Topological embeddings and topological isotopies are locally flat.

We say that a link $L$ in a 3-manifold $Y$ has a \emph{geometric dual} if there is an immersed sphere $G$ in $Y$ meeting $L$ transversely in exactly one point.
By the loop theorem, $G$ may be chosen to be embedded \cite[proof of Lemma~2.7]{Cha-Kim:2023-1}.

\begin{theoremalpha}
  \label{theorem:topological-isotopy-implies-smooth-isotopy}
  Let $M$ be a smooth orientable $4$-manifold with nonempty boundary.
  For any two smooth embeddings of a surface into $M$ that agree on the boundary and whose common boundary has a geometric dual in $\partial M$, topological isotopy implies smooth isotopy.
\end{theoremalpha}

Combined with the light bulb smoothing theorem of~\cite{Cha-Kim:2023-1}, our result yields the following agreement of the topological and smooth isotopy classifications.
For a surface $S$ and a smooth embedding $L$ of $\partial S$ into the boundary of a smooth 4-manifold $M$, write $\pi_0 \Emb^s_L(S,M)$ and $\pi_0 \Emb^t_L(S,M)$ for the sets of smooth and topological isotopy classes of embeddings of $S$ into $M$ that extend~$L$.
For the definitions of the embedding spaces $\Emb^s_L(S,M)$ and $\Emb^t_L(S,M)$, see \cref{subsection:embedding-spaces}.

\begin{corollaryalpha}
  \label{corollary:topological=smooth-for-surfaces}
  Let $M$ be a smooth orientable 4-manifold and let $S$ be a surface with nonempty boundary.
  Suppose that $L\colon \partial S \to \partial M$ is a smooth embedding with a geometric dual in~$\partial M$.
  Then the natural map $\pi_0 \Emb^s_L(S,M) \to \pi_0 \Emb^t_L(S,M)$ is a bijection.
\end{corollaryalpha}

\begin{proof}
  Injectivity follows from \cref{theorem:topological-isotopy-implies-smooth-isotopy}.
  Surjectivity follows from the light bulb smoothing theorem \cite[Theorem~C]{Cha-Kim:2023-1} in the embedding form of \cref{theorem:light-bulb-smoothing-for-embedding} in this paper.
\end{proof}

The analogues hold for unparametrized surfaces in a 4-manifold.
See \cref{corollary:unparametrized-topological-isotopy-implies-smooth-isotopy,corollary:unparametrized-topological=smooth-for-surfaces}.

For disks, \cref{corollary:topological=smooth-for-surfaces} was established in \cite[Corollary~A]{Cha-Kim:2023-1}.

\paragraph{Related work}

For knots, links, and surfaces in a 3-manifold, smooth and topological isotopy give the same classification~\cite{Moise:1954-1}.
For smooth embeddings of a compact smooth $k$-manifold into a smooth $n$-manifold with $2n>3(k+1)$, topological isotopy implies smooth isotopy by Haefliger's work~\cite{Haefliger:1962/1963-1}.
For embeddings of surfaces, this range includes every ambient dimension $n \ge 5$.
However, for embeddings of high-dimensional manifolds, the implication fails in general.
For $r\ge 1$, Haefliger constructed smoothly nonisotopic submanifolds diffeomorphic to $S^{4r-1}$ in $S^{6r}$ which are topologically isotopic~\cite{Haefliger:1962-1}.

For surfaces in 4-manifolds, numerous examples of topologically isotopic but smoothly nonisotopic surfaces are known.
For instance, see work of Finashin, Kreck and Viro~\cite{Finashin-Kreck-Viro:1987-1,Finashin-Kreck-Viro:1988-1} (cf.\ \cite[\S1, footnote~2]{Baykur-Sunukjian:2016-1}), Kim and Ruberman~\cite{Kim-Ruberman:2008-1}, Auckly, Kim, Melvin and Ruberman \cite{Auckly-Kim-Melvin-Ruberman:2015-1}, Hayden and Sundberg~\cite{Hayden-Sundberg:2024-1} and Miyazawa~\cite{Miyazawa:2023-1}.
In particular, \cite[Theorem~1.1]{Hayden-Sundberg:2024-1} shows that the boundary geometric dual hypothesis in \cref{theorem:topological-isotopy-implies-smooth-isotopy} cannot be omitted in general.

By Gabai~\cite{Gabai:2020-1}, Schneiderman and Teichner~\cite{Schneiderman-Teichner:2022-1} and our earlier work \cite{Cha-Kim:2023-1}, topological isotopy implies smooth isotopy for spheres with a common smooth framed embedded geometric dual.
Further affirmative results, under suitable hypotheses, are obtained after stabilizing the surfaces or the ambient $4$-manifold.
See, for instance, Baykur and Sunukjian~\cite{Baykur-Sunukjian:2016-1}, Auckly, Kim, Melvin, Ruberman and Schwartz \cite{Auckly-Kim-Melvin-Ruberman-Schwartz:2019-1}, Galvin~\cite{Galvin:2024-1} and Galvin, Orson and Powell~\cite{Galvin-Orson-Powell:2026-1}.
Our result requires neither type of stabilization.

\paragraph{Proof strategy}

Choose a system of pairwise disjoint properly embedded arcs which cuts the surface into a disk.
Our main strategy is to reduce the surface isotopy problem to that of the resulting disk.
For this purpose, it is crucial to understand the smooth and topological behavior of the bicollars of the arcs throughout isotopy;
note that this is a codimension two problem, while the core arcs have codimension three.
Understanding isotopy of the core arcs is insufficient, since the relative disk problem depends on how the surface thickens the core arcs.
We carry out this comparison using restriction fibrations for relevant embedding spaces and the associated homotopy exact sequences.
The resulting low-degree comparisons between smooth and topological embedding spaces allow us to apply the smooth-topological comparison result for disks with boundary geometric dual in~\cite[Corollary~A]{Cha-Kim:2023-1}.

An essential local problem is to understand one-parameter families of thickenings of a fixed arc to 2-dimensional bicollars.
At each point of the arc, the smooth local model has homotopy
type $O(3)/O(2)\cong S^2$, while the topological model is
$\TOP(4,1)/\TOP(4,2)$.
Here $\TOP(n,m)$ is the group of homeomorphisms of $\R^n$ fixing the standard subspace $\R^m=\R^m\times\{0\}\subset \R^n$ pointwise, equipped with the compact-open topology.
We establish that the natural comparison map $S^2\cong O(3)/O(2) \to \TOP(4,1)/\TOP(4,2)$ induces an isomorphism on $\pi_2$ in~\cref{lemma:pi_2TOP(4_1)/TOP(4_2)}.
A key ingredient for this is the vanishing $\pi_2\TOP(4,1)=0$, which we prove in~\cref{lemma:pi2TOP(4_1)}.
To our knowledge, this vanishing has not previously been recorded in the literature.

\subsubsection*{Organization of the paper}

In \cref{section:preliminaries-on-embeddings}, we discuss preliminaries on smooth and topological embedding spaces and on arcs in 4-manifolds.
In \cref{section:topological-stiefel-spaces}, we compute the second homotopy groups of $\TOP(4,1)$ and $\TOP(4,1)/\TOP(4,2)$.
In \cref{section:embeddings-of-surfaces}, we prove \cref{theorem:topological-isotopy-implies-smooth-isotopy} using results in the previous sections.

\subsubsection*{Acknowledgements}

We thank Jianfeng Lin, Boyu Zhang and Yi Xie for helpful discussions on surface embeddings.
This work was supported by the National Research Foundation grant RS-2019-NR039996.
During the preparation of this paper, we used generative AI tools for copyediting and linguistic refinement.
The authors take full responsibility for all content of this paper.

\section{Preliminaries on embedding spaces}
\label{section:preliminaries-on-embeddings}

\subsection{Embedding spaces and restriction fibration}
\label{subsection:embedding-spaces}

In this subsection, we fix our conventions for embedding spaces and recall the restriction fibrations and smooth-topological comparison maps.

\paragraph{Smooth embedding spaces}

Let $V=(V;\partial_-V,\partial_+V)$ be a compact smooth manifold triad:
$V$ is a smooth manifold equipped with a boundary decomposition $\partial V=\partial_-V\cup\partial_+V$ into faces $\partial_-V$ and $\partial_+V$, with corner $\partial_-V\cap\partial_+V=\partial(\partial_-V)=\partial(\partial_+V)$.
Unless stated otherwise, we regard a compact smooth manifold with boundary as a triad with $\partial_-V=\partial V$ and $\partial_+V=\emptyset$.

Let $M$ be a smooth manifold, possibly with boundary.
We assume that embeddings $f\colon V\hookrightarrow M$ are neat along $\partial_-V$, i.e., $f^{-1}(\partial M)=\partial_-V$ and $f$ is transverse to $\partial M$ along $\partial_-V$, including at corner points.

Fix a reference embedding $V\hookrightarrow M$ and identify $V$ with its image, viewing $V\subset M$ as a subspace.
For a subset $A\subset V$, define $\Emb^s(V,M\rel A)$ to be the space of smooth embeddings $f$ such that $f|_{\partial_-V\cup A}=\id_{\partial_-V\cup A}$, equipped with the $C^\infty$ topology.
The interior of $\partial_+V$ may move in the interior of $M$.
Write $\Emb^s(V,M)=\Emb^s(V,M\rel\emptyset)$.

When only a smooth embedding $L\colon\partial_-V\hookrightarrow\partial M$ is prescribed without specifying a reference embedding $V\hookrightarrow M$, define $\Emb^s_L(V,M)$ to be the space of smooth embeddings $f\colon V\hookrightarrow M$ such that $f|_{\partial_-V}=L$, equipped with the $C^\infty$ topology.
This is the smooth embedding space used in \cref{corollary:topological=smooth-for-surfaces}.

\paragraph{Topological embedding spaces}

We use a simplicial model of locally flat embeddings, following the approach in \cite[Appendix~I]{Burghelea-Lashof-Rothenberg:1975-1}.

Let $V=(V;\partial_-V,\partial_+V)$ be a compact topological manifold triad, with local charts respecting the faces and corner, and let $M$ be a topological manifold, possibly with boundary.
We assume that an embedding $f\colon V\hookrightarrow M$ is neat along $\partial_-V$ and locally flat.
We say that $f$ is \emph{neat along $\partial_-V$} if $f^{-1}(\partial M)=\partial_-V$.
Such an embedding is \emph{locally flat} if it admits the standard coordinate models at points in $\inte V$, $\inte\partial_-V$, $\inte\partial_+V$ and $\partial(\partial_-V)$.
See Definition~1.2 in Appendix~I and the remark on p.~113 of \cite{Burghelea-Lashof-Rothenberg:1975-1} for these local models.

Fix a reference embedding $V\hookrightarrow M$ satisfying the preceding conditions, and identify $V$ with its image.
For a subset $A\subset V$, let $\Emb^t_\bullet(V,M\rel A)$ be the simplicial set whose $q$-simplices are level-preserving, locally flat $\Delta^q$-families of embeddings $F\colon V\times\Delta^q\hookrightarrow M\times\Delta^q$, such that if we write $F(x,t)=(f_t(x),t)$, each $f_t$ is neat along $\partial_-V$ and satisfies $f_t|_{\partial_-V\cup A}=\id_{\partial_-V\cup A}$.
The simplicial operators are given by pullback along the standard affine maps between simplices.
For the definition of local flatness for $\Delta^q$-families, refer to \cite[Appendix~I, condition~(2.8) and Section~2(c)]{Burghelea-Lashof-Rothenberg:1975-1} and \cite[p.~64]{Edwards-Kirby:1971-1}, adapted using the face and corner models.
See also \cite[Definitions~A.1 and~A.3]{Gomez-Lopez-Kupers:2022-1} for the special case of $\partial_+V=\emptyset$.

Define $\Emb^t(V,M\rel A)=|\Emb^t_\bullet(V,M\rel A)|$, the geometric realization.
This is the topological embedding space used in this paper.
Its path components are the rel $\partial_-V\cup A$ locally flat isotopy classes of locally flat embeddings.
When $A=\emptyset$, write $\Emb^t_\bullet(V,M)$ and $\Emb^t(V,M)$ for brevity.

For a locally flat embedding $L\colon\partial_-V\hookrightarrow\partial M$, define $\Emb^t_{L,\bullet}(V,M)$ by the same construction with $A=\emptyset$, replacing the boundary condition $f_t|_{\partial_-V}=\id_{\partial_-V}$ by $f_t|_{\partial_-V}=L$, without fixing a reference embedding $V\hookrightarrow M$.
Let $\Emb^t_L(V,M)=|\Emb^t_{L,\bullet}(V,M)|$.
This is the topological embedding space used in \cref{corollary:topological=smooth-for-surfaces}.

\paragraph{Restriction fibrations}

Let $V$ be a compact manifold triad, and fix a locally flat reference embedding $V\subset M$ that is neat along $\partial_-V$.
Suppose that $A\subset V$ is a compact manifold triad satisfying $\partial_-A=A\cap\partial_-V$.
We assume that, at each point of $A$, there is a single ambient chart in which $V$ and $A$ are given by standard subsets defined by coordinate vanishing and nonnegativity conditions.
In the smooth category, the manifolds, reference embedding, and these charts are assumed to be smooth.
Let $B\subset A$ be any subset.

In the smooth category, the restriction map $\Emb^s(V,M\rel B)\to\Emb^s(A,M\rel B)$ is a Serre fibration \cite[Appendix~I, Theorem~4.1 and Remark~(b)] {Burghelea-Lashof-Rothenberg:1975-1};
see also \cite[Chapter~II, Sections~2.2.1--2.2.2, Theorem~5 and Corollary~2] {Cerf:1961-1}.
The fiber over the reference inclusion $\id_A$ is
$\Emb^s(V,M\rel A)$.

In the topological category, restriction gives rise to a Kan fibration $\Emb^t_\bullet(V,M\rel B) \to \Emb^t_\bullet(A,M\rel B)$ with fiber $\Emb^t_\bullet(V,M\rel A)$ over~$\id_A$.
When $B=\emptyset$, this follows from parametrized isotopy extension \cite[Theorem~6.17]{Siebenmann:1972-1}, by adapting the argument preceding \cite[Theorem~A.4]{Gomez-Lopez-Kupers:2022-1} to our models with faces and corner.
The general rel $B$ case follows from the $B=\emptyset$ case without requiring additional assumptions on~$B$, since the rel $B$ restriction map is the pullback of $\Emb^t_\bullet(V,M) \to \Emb^t_\bullet(A,M)$ along $\Emb^t_\bullet(A,M\rel B)\hookrightarrow\Emb^t_\bullet(A,M)$.
We also remark that when $\partial_+V=\partial_+A=\emptyset$, the $B=\emptyset$ case is covered by \cite[Appendix~I, Theorem~4.15]{Burghelea-Lashof-Rothenberg:1975-1}.
It follows that geometric realization gives a Serre fibration $\Emb^t(V,M\rel B)\to\Emb^t(A,M\rel B)$ with fiber $\Emb^t(V,M\rel A)$, by \cite[Theorem~I.10.10 and Proposition~I.2.4]{Goerss-Jardine:1999-1}.

Consequently, for $c=s$,~$t$, restriction gives rise to a fiber sequence
\[
  \Emb^c(V,M\rel A) \to \Emb^c(V,M\rel B) \to \Emb^c(A,M\rel B)
\]
and the associated homotopy long exact sequence.
All embedding spaces are based at the reference inclusions.

\paragraph{Comparison maps}

Suppose that $V$ and $M$ are smooth, and $A\subset V$ is a triad satisfying the hypothesis for the restriction fibration.
Define $\Emb^s_\bullet(V,M \rel A)$ by requiring the embedding $V\times\Delta^q \hookrightarrow M\times\Delta^q$ in the definition of $\Emb^t_\bullet(V,M \rel A)$ to be smooth.
Then there is a natural simplicial map $j\colon \Emb^s_\bullet(V,M \rel A) \to \Emb^t_\bullet(V,M \rel A)$.
Furthermore, the evaluation map $e\colon |\Emb^s_\bullet(V,M \rel A)|\to\Emb^s(V,M \rel A)$ defined by $[F,t]\mapsto F_t$, where $F(x,t)=(F_t(x),t)$, is a weak homotopy equivalence by the proof of \cite[Appendix I, Proposition~6.1]{Burghelea-Lashof-Rothenberg:1975-1}, applied relative to $\partial_-V \cup A$.
Thus, $(|j|)_*\circ(e_*)^{-1}$ defines a natural comparison map $\pi_k\Emb^s(V,M \rel A)\to\pi_k\Emb^t(V,M \rel A)$.

For the restriction fiber sequence associated with $B\subset A\subset V$, the comparison maps are compatible with the homotopy long exact sequences when each of $(V,A)$, $(V,B)$, and $(A,B)$ satisfies the hypothesis for the comparison map.

\begin{remark}[Collar boundary conditions]
  \label{remark:collar-boundary-conditions}
  In our convention, embeddings of $V$ agree with the reference inclusion on $\partial_-V$ pointwise.
  It is also useful to consider embedding spaces requiring agreement on a collar.
  Choose collars of $\partial_-V$ and $\partial M$ compatible with the reference inclusion and the face structures.
  For $c=s$,~$t$, let $\Emb^c_{\coll,\bullet}(V,M)$ be the simplicial subset of $\Emb^c_\bullet(V,M)$ consisting of families that agree with the reference inclusion on some subcollar, whose width is uniform over each parameter simplex.
  Write $\Emb^c_{\coll}(V,M)=|\Emb^c_{\coll,\bullet}(V,M)|$.

  For $c=s$, define $\Emb^s_{\coll}(V,M)\to\Emb^s(V,M)$ to be the composition of the inclusion into $|\Emb^s_\bullet(V,M)|$ and the evaluation map~$e$.
  For $c=t$, the simplicial inclusion induces $\Emb^t_{\coll}(V,M)\to\Emb^t(V,M)$.
  These comparison maps are weak homotopy equivalences.
  For $c=s$, it follows by combining the arguments of \cite[Proposition~2.9]{Kosanovic-Teichner:2024-1} and \cite[Appendix~I, Proposition~6.1]{Burghelea-Lashof-Rothenberg:1975-1} and passing to shrinking collar widths.
  For $c=t$, apply the collar-insertion argument of \cite[proof of Theorem~2.21]{Friedl-Nagel-Orson-Powell:2025-1} to finite simplicial parameter pairs; see \cite[proof of Lemma~2.2, step~(i)]{Kupers:2015-1} for a parametrized version.

  For a triad $A\subset V$ satisfying the hypothesis for the restriction fibration, choose compatible collars and define $\Emb^c_{\coll}(V,M\rel A)$ similarly, fixing $A$ pointwise.
  The parametrized isotopy extension argument applied relative to a small common boundary collar yields a homotopy fiber sequence $\Emb^c_{\coll}(V,M\rel A)\to\Emb^c_{\coll}(V,M)\to\Emb^c_{\coll}(A,M)$.
  Comparing this with the corresponding homotopy fiber sequence for $\Emb^c$, it follows that $\Emb^c_{\coll}(V,M\rel A)\to\Emb^c(V,M\rel A)$ is a weak homotopy equivalence.
  These comparison maps are compatible with restriction and with smooth-to-topological comparison.
\end{remark}

\subsection{Arcs in 4-manifolds}

Let $M$ be a smooth $4$-manifold with nonempty boundary.
Let $K=\bigsqcup_{i=1}^m K_i$ be the disjoint union of $m$ unit intervals $K_i \cong I$.
Fix a smooth neat embedding $\iota\colon K \hookrightarrow M$ and view $K\cong \iota(K)\subset M$ as a submanifold of~$M$.
Let $\Emb^s(K,M)$ and $\Emb^t(K,M)$ be the smooth and topological embedding spaces defined as above, viewing $K$ as a triad with $\partial_-K = \partial K$ and $\partial_+K=\emptyset$.

\begin{lemma}
  \label{lemma:pi0-arcs-comparison}
  The map $\pi_0\Emb^s(K,M)\to\pi_0\Emb^t(K,M)$ is injective.
\end{lemma}

\begin{proof}
  Suppose that $f_0$, $f_1\colon K\hookrightarrow M$ are smooth embeddings representing the same element of $\pi_0\Emb^t(K,M)$.
  Then there is a homotopy $F\colon K\times I\to M$ rel $\partial K$ between $f_0$ and $f_1$, with $F^{-1}(\partial M)=\partial K\times I$.
  By a local adjustment and relative smooth approximation, we may assume that $F$ is smooth and each slice $f_t(x)=F(x,t)$ is a neat embedding near~$\partial K$.

  By relative general position, we may further assume that each $f_t$ is an immersion and that $F(x,t)\ne F(y,t)$ whenever $x\ne y$.
  Indeed, the condition $d(f_t)_x=0$ has codimension~$4$, whereas $K\times I$ has dimension~$2$.
  The condition $F(x,t)=F(y,t)$ for $x\ne y\in K\sm \partial K$ has codimension~$4$, whereas the space of such triples $(x,y,t)$ has dimension~$3$.

  Since $K$ is compact and since $f_t$ is an injective immersion, $f_t$ is a smooth embedding for each~$t$.
  Therefore, $\{f_t\}$ is a smooth isotopy rel $\partial K$ between $f_0$ and~$f_1$.
\end{proof}

\begin{lemma}
  \label{lemma:pi1-arcs-comparison}
  The map $\pi_1\Emb^s(K,M)\to\pi_1\Emb^t(K,M)$ is surjective.
\end{lemma}

\begin{proof}
  Let $\alpha \in \pi_1\Emb^t(K,M)$ be an arbitrary element.
  Since $\Emb^t_\bullet(K,M)$ is a Kan complex \cite[Lemmas~3.5 and~3.6]{Kupers:2015-1}, $\alpha=[f]$ for some $1$-simplex $f=\{f_t\colon K\to M\}_{0\le t\le 1}$ with $f_0 = f_1 = \id_K$.
  By isotopy extension \cite{Edwards-Kirby:1971-1}, \cite[Theorem~3.9 and Remark~3.10]{Kupers:2015-1}, there is an ambient isotopy $\{h_t\colon M \to M\}_{0\le t\le 1}$ of $M$ rel~$\partial M$ such that $h_0=\id$ and $h_t|_K = f_t$.

  For the smooth arcs $f_0(K)$, choose a smooth open tubular neighborhood $N=\bigsqcup_{i=1}^m N_i$, where $N_i \cong K_i\times \R^3$.
  Let $U_t = h_t(N)$, which is a (topological) tubular neighborhood of the topological arcs~$f_t(K)$.

  By compactness, we can choose $0=t_0<t_1<\cdots<t_r=1$ such that $f_t(K)\subset U_{t_j}$ for $t\in [t_{j-1},t_j]$.
  Write $U_j = U_{t_j}$ for brevity.
  For $1\le j\le r-1$, choose a smaller topological open tubular neighborhood $C_j \subset U_j\cap U_{j+1}$ of~$f_{t_j}(K)$.
  Since each component of $C_j$ is connected, we can choose a smooth neat embedding $g_j\colon K\hookrightarrow C_j$ such that $g_j|_{\partial K} = \id_{\partial K}$.
  Let $g_0=f_0$ and $g_r=f_1$.

  For each $j$, both $g_{j-1}$ and $g_j$ lie in~$U_j$, and each component of $U_j$ is simply connected.
  Thus, $g_{j-1}$ and $g_j$ are homotopic rel~$\partial K$ in~$U_j$, and by perturbing a homotopy, one obtains a smooth isotopy $\gamma_j\colon g_{j-1}\approx g_j$ in~$U_j$, rel~$\partial K$.
  After making each $\gamma_j$ stationary near the parameter endpoints by reparametrization, the composition $\gamma=\gamma_1*\cdots*\gamma_r$ is a smooth loop in $\Emb^s(K,M)$ based at~$\id_K$.

  It is known that the space $\Emb^t_{\coll}(D^1, D^4)$ is contractible \cite[Lemma~3.1 and Proposition~3.4(b)]{Salvatore-Turchin:2026-1}.
  Thus, by \cref{remark:collar-boundary-conditions}, $\Emb^t(D^1,D^4)$ is contractible.
  Note that $\Emb^t(D^1,D^4) \cong \Emb^t(D^1,D^1\times D^3) \cong \Emb^t(D^1,D^1\times\inte(D^3))$.
  Thus, $\Emb^t(K,C_j) \cong \prod^m \Emb^t(D^1,D^1\times \inte(D^3))$ is contractible.
  Similarly, $\Emb^t(K,U_j)$ is contractible.

  Since $\pi_0 \Emb^t(K,C_j) = 0$, we may choose locally flat isotopies $\lambda_j \colon f_{t_j} \approx g_j$ rel $\partial K$ in $C_j$ for each $j=1,\ldots,r-1$.
  Let $\lambda_0$ and $\lambda_r$ be constant isotopies.

  Let $a_j$ be the affine reparametrization of the restriction $f|_{[t_{j-1},t_j]}$.
  We have $[f]=[a_1]\cdots[a_r]$ in the fundamental groupoid of $\Emb^t(K,M)$.
  For each $j$, the loop $a_j*\lambda_j*(\gamma_j)^{-1}* (\lambda_{j-1})^{-1}$ lies in $\Emb^t(K,U_j)$, which is simply connected.
  Thus, we have $[a_j] = [\lambda_{j-1}*\gamma_j*(\lambda_j)^{-1}]$ in the fundamental groupoid of $\Emb^t(K,U_j)$.
  Multiplying these identities yields $[f]=[\gamma]$ in $\pi_1\Emb^t(K,M)$.
  It follows that $[f]$ lies in the image of $\pi_1\Emb^s(K,M)\to\pi_1\Emb^t(K,M)$.
\end{proof}

\begin{remark}
  Using similar methods, one can also show that $\pi_0 \Emb^s(K,M) \to \pi_0 \Emb^t(K,M)$ is a bijection and $\pi_1 \Emb^s(K,M) \to \pi_1 \Emb^t(K,M)$ is an isomorphism.
  We do not use these facts in this paper.
\end{remark}

\section{Homotopy groups of \texorpdfstring{$\TOP(4,1)$}{TOP(4,1)} and \texorpdfstring{$\TOP(4,1)/\TOP(4,2)$}{TOP(4,1)/TOP(4,2)}}
\label{section:topological-stiefel-spaces}

The goal of this section is to prove the following results.

\begin{lemma}
  \label{lemma:pi2TOP(4_1)}
  $\pi_2\TOP(4,1)=0$.
\end{lemma}

\begin{lemma}
  \label{lemma:pi_2TOP(4_1)/TOP(4_2)}
  The natural map
  \[
    S^2 \cong \O(3)/\O(2) \to \TOP(4,1)/\TOP(4,2)
  \]
  induces an isomorphism on $\pi_2$.
  In particular, $\pi_2(\TOP(4,1)/\TOP(4,2))\cong\Z$.
\end{lemma}

In the proof, we use Stiefel spaces.
For $\CAT = \O$, $\PL$ and $\TOP$, define the $\CAT$-Stiefel space by $V^{\CAT}_{n,m}=\CAT(n)/\CAT(n,m)$, where $\O(n)$ is the orthogonal group, $\PL(n)$ and $\TOP(n)$ are the groups of PL homeomorphisms and homeomorphisms of $\R^n$ fixing the origin, respectively, and $\CAT(n,m)\subset \CAT(n)$ is the subgroup fixing $\R^m\subset\R^n$ pointwise.
Note that $\CAT(n,0)=\CAT(n)$ and $\O(n,m) \cong \O(n-m)$.
The natural homomorphisms $\O(n) \to \PL(n) \to \TOP(n)$ induce comparison maps $V^{\O}_{n,m} \to V^{\PL}_{n,m} \to V^{\TOP}_{n,m}$.

We remark that the Stiefel spaces also admit the standard germ and simplicial models, which have the same weak homotopy types as the models we used above.

The natural maps form a homotopy fiber sequence $\CAT(n,m) \to \CAT(n) \to V^{\CAT}_{n,m}$.
For $\CAT=\O$, this follows since $\O(n-m)\subset \O(n)$ is a closed Lie subgroup.
For $\CAT=\PL$,~$\TOP$, it follows from \cite[Appendix~I, Theorems~4.10 and~4.14]{Burghelea-Lashof-Rothenberg:1975-1} using simplicial models.

In addition, a natural inclusion $\CAT(n,m)\subset \CAT'(n',m')$ gives rise to a homotopy fiber sequence $\CAT(n,m)\to \CAT'(n',m') \to \CAT'(n',m')/\CAT(n,m)$, and thus to an associated homotopy long exact sequence \cite[Section~3.4, Remark~3.7]{Salvatore-Turchin:2026-1}.

\begin{proof}[Proof of \cref{lemma:pi2TOP(4_1)}]
  Let $F^{\PL}_{4,1}$ be the homotopy fiber of $V^{\O}_{4,1}\to V^{\PL}_{4,1}$.
  By \cite[Theorem~2.3, equation~(2.4)]{Salvatore-Turchin:2026-1}, $\Omega F^{\PL}_{4,1}$ is weakly homotopy equivalent to the space of smooth embeddings $D^1\hookrightarrow D^4$ agreeing with the standard inclusion near~$\partial D^1$.
  This space is simply connected by \cite[Proposition~3.9(2)]{Budney:2008-1}.
  Thus $\pi_2F^{\PL}_{4,1}=0$.
  By the homotopy long exact sequence of $F^{\PL}_{4,1}\to V^{\O}_{4,1}\to V^{\PL}_{4,1}$, it follows that $\pi_3V^{\O}_{4,1}\to\pi_3V^{\PL}_{4,1}$ is surjective.

  By \cite[Theorem~1(2)]{Kurata:1976-1} and the preceding model identifications, the comparison map $V^{\PL}_{4,1}\to V^{\TOP}_{4,1}$ is a weak homotopy equivalence.
  By compatibility of the comparison maps, it follows that the map $\pi_3 V^{\O}_{4,1}\to \pi_3 V^{\TOP}_{4,1}$ is surjective.

  Since $\pi_2\O(3) = 0$, the homotopy long exact sequence of $\O(3) \to \O(4) \to V^{\O}_{4,1}$ shows that $\pi_3 \O(4) \to \pi_3 V^{\O}_{4,1}$ is surjective.
  Thus, the composition $\pi_3 \O(4) \to \pi_3 V^{\O}_{4,1} \to \pi_3 V^{\TOP}_{4,1}$ is surjective.
  From the commutative diagram
  \[
    \begin{tikzcd}
      \pi_3\O(4) \ar[r] \ar[d]
      & \pi_3\TOP(4) \ar[d]\\
      \pi_3V^{\O}_{4,1} \ar[r]
      & \pi_3V^{\TOP}_{4,1}
    \end{tikzcd}
  \]
  it follows that $\pi_3\TOP(4) \to \pi_3V^{\TOP}_{4,1}$ is surjective.

  We also need the fact that $\pi_2\TOP(4)=0$.
  Indeed, the natural map $\TOP(4)/\O(4)\to\TOP/\O$ is $3$-connected \cite[Corollary~2.2.3]{Quinn:1982-1}, and $\TOP/\O$ is $2$-connected \cite[Essay~V, Theorem~(**)]{Kirby-Siebenmann:1977-1}.
  It follows that $\TOP(4)/\O(4)$ is $2$-connected.
  Since $\pi_2\O(4)=0$, the homotopy long exact sequence of $\O(4) \to \TOP(4) \to \TOP(4)/\O(4)$ gives $\pi_2\TOP(4)=0$.

  The homotopy exact sequence of $\TOP(4,1)\to\TOP(4)\to V^{\TOP}_{4,1}$ contains
  \[
    \pi_3 \TOP(4) \to \pi_3 V^{\TOP}_{4,1} \to \pi_2 \TOP(4,1) \to \pi_2 \TOP(4).
  \]
  The first map is surjective and the last group is zero.
  It follows that $\pi_2\TOP(4,1)=0$.
\end{proof}

\begin{proof}[Proof of \cref{lemma:pi_2TOP(4_1)/TOP(4_2)}]
  Consider the following commutative diagram:
  \[
    \begin{tikzcd}[column sep=small]
      0=\pi_2 \O(3) \ar[r] &
      \pi_2 (\O(3) / \O(2)) \ar[r] \ar[d] &
      \pi_1 \O(2) \ar[r] \ar[d] &
      \pi_1 \O(3) \ar[d]
      \\
      0=\pi_2 \TOP(4,1) \ar[r] &
      \pi_2 (\TOP(4,1) / \TOP(4,2)) \ar[r] &
      \pi_1 \TOP(4,2) \ar[r] &
      \pi_1 \TOP(4,1)
    \end{tikzcd}
  \]

  Note that $\pi_2 \TOP(4,1) = 0$ by Lemma~\ref{lemma:pi2TOP(4_1)}.
  It is known that $\TOP(4,2)/\O(2)$ is 2-connected; see~\cite[proof of Theorem~2.1]{Salvatore-Turchin:2026-1}.
  In particular, $\pi_1 \O(2) \to \pi_1 \TOP(4,2)$ is an isomorphism.
  The map $\pi_1\O(3)\to\pi_1\TOP(4,1)$ is injective, since its composition with the map $\pi_1\TOP(4,1)\to\pi_1\TOP(4)$ agrees with the isomorphism
  \[
    \pi_1\O(3)\xrightarrow{\cong}\pi_1\O(4)\xrightarrow{\cong}\pi_1\TOP(4).
  \]
  Here, the second map is an isomorphism since $\TOP(4)/\O(4)$ is $2$-connected by \cite[Corollary~2.2.3]{Quinn:1982-1} and \cite[Essay~V, Theorem~5.3(**)]{Kirby-Siebenmann:1977-1}.

  By exactness and the above two assertions, the leftmost vertical map is an isomorphism.
\end{proof}

\section{Embeddings of surfaces into 4-manifolds}
\label{section:embeddings-of-surfaces}

\subsection{Light bulb smoothing and isotopy of disks}

In this section, we will use the following embedding versions of two results in~\cite{Cha-Kim:2023-1}.

\begin{theorem}[Light bulb smoothing for embeddings {\cite[Theorem~C]{Cha-Kim:2023-1}}]
  \label{theorem:light-bulb-smoothing-for-embedding}
  Let $M$ be a smooth 4-manifold, and $f\colon S \hookrightarrow M$ be a neat topological embedding of a surface $S$ such that $f|_{\partial S}\colon \partial S\hookrightarrow \partial M$ is smooth and has a geometric dual in~$\partial M$.
  Then $f$ is topologically isotopic rel $\partial S$ to a smooth embedding.
\end{theorem}

Indeed, \cite[Theorem~C]{Cha-Kim:2023-1} gives the analogous statement for unparametrized surfaces in~$M$:
if $S \subset M$ is a topological neat surface such that $\partial S$ is smooth and has a geometric dual in~$\partial M$, then $S$ is topologically isotopic rel $\partial S$ to a smooth surface in~$M$.
\Cref{theorem:light-bulb-smoothing-for-embedding} readily follows from this, using the fact that a homeomorphism between surfaces which restricts to a diffeomorphism of the boundary is topologically isotopic rel boundary to a diffeomorphism.

\begin{remark}
  The light bulb smoothing theorem holds for nonorientable surfaces and nonorientable ambient 4-manifolds as well, although the introduction of \cite{Cha-Kim:2023-1} assumes that all manifolds are oriented.
  In fact, the proof of \cite[Theorem~C]{Cha-Kim:2023-1} does not use orientability.
\end{remark}

Similarly, the following is an embedding version of \cite[Corollary~A]{Cha-Kim:2023-1}.

\begin{theorem}[{\cite[Corollary~A]{Cha-Kim:2023-1}} for embeddings]
  \label{theorem:topological=smooth-for-disks}
  Let $M$ be a smooth orientable 4-manifold and $K\colon \partial D^2 \hookrightarrow \partial M$ be a smooth knot which has a geometric dual in~$\partial M$.
  Then the natural map $\pi_0 \Emb^s_K(D^2,M) \to \pi_0 \Emb^t_K(D^2,M)$ is a bijection.
\end{theorem}

\subsection{Topological and smooth isotopy of surfaces}

Let $M$ be a smooth $4$-manifold with nonempty boundary, and let $S$ be a compact connected surface, possibly nonorientable, with nonempty boundary.
Let $K\subset S$ be the union of pairwise disjoint smoothly embedded neat arcs, such that cutting $S$ open along $K$ gives a disk.
Note that $K$ has $b_1(S)$ components, where $b_1(S) =  1 - \chi(S)$ is the first Betti number of~$S$.
Fix a bicollar neighborhood $U$ of $K$ in $S$, together with an identification $(U,K)\cong (K\times D^1,K\times 0)$.

To define embedding spaces, regard $S$ and $K$ as triads with $\partial_+S = \emptyset$, $\partial_+K = \emptyset$.
View $U\cong K\times D^1$ as a triad with $\partial_- U = \partial K \times D^1$, $\partial_+U = K\times \partial D^1$. 
Fix a smooth neat embedding $S\hookrightarrow M$, and under this, regard $S$, $K$, and $U$ as subsets of~$M$.

\begin{lemma}
  \label{lemma:embedding-spaces-bicollar-rel-core}
  There is a homotopy commutative diagram
  \[
    \begin{tikzcd}[column sep=normal]
      \Emb^s(U,M\rel K)
      \arrow[r, "\gamma^s", "\simeq"']
      \arrow[d]
      &
      \bigl(\Omega(\O(3)/\O(2))\bigr)^{b_1(S)}
      \arrow[d]
      \\
      \Emb^t(U,M\rel K)
      \arrow[r, "\gamma^t", "\simeq"']
      &
      \bigl(\Omega(\TOP(4,1)/\TOP(4,2))\bigr)^{b_1(S)},
    \end{tikzcd}
  \]
  where the horizontal maps $\gamma^s$ and $\gamma^t$ are weak homotopy equivalences and the vertical maps are the natural maps.
\end{lemma}

\begin{proof}
  We work in the pointed homotopy category.
  We first pass to collar-fixed embedding spaces, then to immersion spaces, and finally compute their homotopy types using immersion theory.

  By \cref{remark:collar-boundary-conditions}, there is a natural weak homotopy equivalence $\Emb^c_{\coll}(U,M\rel K) \to \Emb^c(U,M\rel K)$ which is compatible with smooth-to-topological comparison.
  Therefore, it suffices to show the conclusion for $\Emb^c_{\coll}(U,M\rel K)$ in place of $\Emb^c(U,M\rel K)$, $c=s$, $t$.

  For $c=s$,~$t$, let $\Imm^c_{\coll}(X,Y\rel A)$ denote the realization of the simplicial model of immersions, with the same collar and pointwise relative conditions as $\Emb^c_{\coll}(X,Y\rel A)$.
  A $q$-simplex is a level-preserving $\Delta^q$-family of immersions $F\colon X\times \Delta^q \looparrowright Y\times \Delta^q$;
  $F$ is smooth for $c=s$, and $F$ is locally a locally flat family of embeddings for $c=t$.

  Let $d=b_1(S)$, and write $K=K_1\sqcup\cdots\sqcup K_d$ and $U=U_1\sqcup\cdots\sqcup U_d$.
  Write $\id$ for the reference inclusions $K\subset U \subset M$.
  Choose disjoint smooth neighborhoods $W_i$ of $U_i$ with product identifications $(W_i,U_i,K_i)\cong I\times(\R^3,D^1\times0,0)$ compatible with collars of $M$, $U_i$, and~$K_i$.
  Let $W=\bigsqcup_iW_i$.
  Shrinking $U$ toward $K\cup\partial_-U$ gives weak homotopy equivalences
  \[
    \Emb^c_{\coll}(U,M\rel K) \xleftarrow{\simeq} \Emb^c_{\coll}(U,W\rel K)
    \xrightarrow{\simeq}
    \Imm^c_{\coll}(U,W\rel K).
  \]
  Indeed, any finite collection of $\Delta^q$-families has a common fixed boundary collar and a common neighborhood of $K\cup\partial_-U$ mapped to~$W$;
  a family of immersions restricts to embeddings on a sufficiently small such neighborhood, cf.~\cite[p.~145]{Lashof:1976-1}.

  Observe that $\Imm^c_{\coll}(U,W\rel K)$ is the product of $\Imm^c_{\coll}(U_i,W_i\rel K_i)$.
  We will compute each factor.

  Let $P=(0,1)\times(-1,1)$, $Q=(0,1)\times\R^3$ and $L=(0,1)\times0$, and view them as subsets of $W_i = I\times\R^3$.
  For $X=P,L$, write $\Imm^c_{\mathrm{end}}(X,Q)$ for the realization of the analogous simplicial immersion model requiring agreement with $\id$ at all points $(t,x)$ with $t\in (0,\epsilon)\cup(1-\epsilon,1)$, for some $\epsilon>0$ uniform over each parameter simplex.
  The notations rel~$K_i$ and rel~$L$ additionally mean, respectively, that $K_i$ and $L$ are fixed pointwise.
  Write $V^s_{4,r}=V^{\O}_{4,r}$ and $V^t_{4,r}=V^{\TOP}_{4,r}$, and let $q_c\colon V^c_{4,2}\to V^c_{4,1}$ be the natural map.

  For $c=s$, $t$, consider the following homotopy commutative diagram:
  \[
    \begin{tikzcd}[
        column sep=small,
        row sep=4ex,
      ]
      \Imm^c_{\coll}(U_i,W_i\rel K_i) \ar[r] \ar[d,"\simeq"']
      & \Imm^c_{\coll}(U_i,W_i) \ar[r] \ar[d,"\simeq"']
      & \Imm^c_{\coll}(K_i,W_i) \ar[d,"\simeq"']
      \\
      \Imm^c_{\mathrm{end}}(P,Q\rel L) \ar[r] \ar[d,"\simeq"']
      & \Imm^c_{\mathrm{end}}(P,Q) \ar[r] \ar[d,"\simeq"']
      & \Imm^c_{\mathrm{end}}(L,Q) \ar[d,"\simeq"']
      \\
      \Omega\hofiber_*(q_c) \ar[r]
      & \Omega V^c_{4,2} \ar[r,"\Omega q_c"]
      & \Omega V^c_{4,1}.
    \end{tikzcd}
  \]

  The vertical arrows from the first row are restrictions to the interiors.
  These are homotopy equivalences.
  For the right arrow, extension by $\id_{\partial K_i}$ gives an inverse.
  For the middle and left arrows, fix $0<a<1$ and define $D_a(t,y)=(t,ay)$ on $I\times\R^3$.
  A homotopy inverse sends $f$ to the map $(t,x)\mapsto D_{1/a}(f(t,ax))$ for $(t,x)\in (0,1)\times D^1$, extended by~$\id_{\partial I\times D^1}$.
  Varying the scale from $a$ to $1$ in the two compositions gives the required homotopies.

  The second row is a homotopy fiber sequence by the restriction theorem for immersions relative to the prescribed end regions;
  in the topological category, see \cite[Appendix, Theorem~3]{Rourke-Sanderson:1970-1}.
  Consequently, the first row is also a homotopy fiber sequence.

  The relative Smale--Hirsch theorem and its topological analogue give a weak homotopy equivalence from $\Imm^c_{\mathrm{end}}(P,Q)$ to the corresponding formal immersion space; see \cite{Hirsch:1959-1} and \cite[Appendix, Corollary~1]{Rourke-Sanderson:1970-1}.
  The formal immersion space is weakly homotopy equivalent to $\Map((I,\partial I), (V^c_{4,2},*)) \simeq \Omega V^c_{4,2}$.
  This gives the weak homotopy equivalence $\Imm^c_{\mathrm{end}}(P,Q) \to \Omega V^c_{4,2}$.
  The weak homotopy equivalence $\Imm^c_{\mathrm{end}}(L,Q) \to \Omega V^c_{4,1}$ is defined similarly.
  These equivalences identify restriction with $\Omega q_c$ up to pointed homotopy.
  Therefore, comparison of the homotopy fibers of the second and third rows shows that $\Imm^c_{\mathrm{end}}(P,Q\rel L)\to\Omega\hofiber_*(q_c)$ is a weak homotopy equivalence.
  Composing this with the upper left vertical arrow gives
  \[
    \Imm^c_{\coll}(U_i,W_i\rel K_i) \simeq \Omega\hofiber_*(q_c).
  \]

  For $c=s$ and $t$, $\hofiber_*(q_c)$ is weakly homotopy equivalent to $\O(4,1)/\O(4,2)\cong\O(3)/\O(2)$ and $\TOP(4,1)/\TOP(4,2)$, respectively.
  Taking products over the components of $K$ gives $\gamma^s$ and $\gamma^t$.

  All involved maps are compatible up to pointed homotopy with smooth-to-topological comparison; see \cite[Section~6.1 and Remark~7.4]{Salvatore-Turchin:2026-1} for the immersion-theoretic compatibility.
  Thus, the induced map on homotopy fibers agrees with the natural map $\O(3)/\O(2)\to\TOP(4,1)/\TOP(4,2)$ up to pointed homotopy.
  Looping and taking products give the asserted homotopy commutativity.
\end{proof}

\begin{lemma}
  \label{lemma:pi1-embedding-bicollar-comparison}
  The natural map
  \[
    \pi_1\Emb^s(U,M\rel K) \to \pi_1\Emb^t(U,M\rel K)
  \]
  is an isomorphism.
  In particular,
  \[
    \pi_1\Emb^s(U,M\rel K) \cong \pi_1\Emb^t(U,M\rel K) \cong \Z^{b_1(S)}.
  \]
\end{lemma}

\begin{proof}
  The conclusion follows immediately from Lemma~\ref{lemma:embedding-spaces-bicollar-rel-core} and Lemma~\ref{lemma:pi_2TOP(4_1)/TOP(4_2)}.
\end{proof}

\begin{lemma}
  \label{lemma:pi0-surface-rel-core-injective}
  Suppose that $M$ is orientable and that $\partial S$ has a geometric dual in~$\partial M$.
  Then the natural map
  \begin{equation}
    \pi_0\Emb^s(S,M\rel K) \to \pi_0\Emb^t(S,M\rel K)
    \label{equation:pi0_S_M_relK}
  \end{equation}
  is injective.
\end{lemma}

\begin{proof}
  For $c=s$,~$t$, restriction to $U$ yields a fibration
  \[
    \Emb^c(S,M\rel U) \to \Emb^c(S,M\rel K) \to \Emb^c(U,M\rel K).
  \]
  The associated homotopy long exact sequences form the following commutative diagram.
  \[
    \begin{tikzcd}[
        column sep=9pt,
        ampersand replacement=\&,
        cells={nodes={inner xsep=1.5pt}}
      ]
      \pi_1\Emb^s(U,M\rel K) \ar[r] \ar[d]
      \& \pi_0\Emb^s(S,M\rel U) \ar[r] \ar[d]
      \& \pi_0\Emb^s(S,M\rel K) \ar[r] \ar[d]
      \& \pi_0\Emb^s(U,M\rel K) \ar[d]
      \\
      \pi_1\Emb^t(U,M\rel K) \ar[r]
      \& \pi_0\Emb^t(S,M\rel U) \ar[r]
      \& \pi_0\Emb^t(S,M\rel K) \ar[r]
      \& \pi_0\Emb^t(U,M\rel K)
    \end{tikzcd}
  \]

  \begin{claim}
    The comparison map $\pi_0\Emb^s(S,M\rel U) \to \pi_0\Emb^t(S,M\rel U)$ is a bijection.
  \end{claim}

  To verify the claim, choose a geometric dual sphere $G\subset\partial M$ whose intersection with $\partial S$ avoids $\partial K$.
  Choose a sufficiently small smooth closed tubular neighborhood $\nu(K)$ adapted to $S$, disjoint from $G$, and satisfying $\nu(K)\cap S\subset\inte_S(U)$.
  Let $X=\overline{M \sm \nu(K)}$ be the exterior.
  Let $D = \overline{S\sm \nu(K)} = S\cap X$, which is a disk, and let $N=\overline{U\sm\nu(K)}$.

  Define $\alpha^c\colon \Emb^c(D,X\rel N)\to \Emb^c(S,M \rel U)$ by $f\mapsto f\cup \id_{S\cap \nu(K)}$, for $c=s$,~$t$.
  This induces a bijection on~$\pi_0$.
  For an embedding $f \in \Emb^c(S,M \rel U)$, radial expansion near $K$ (and corresponding reparametrization of~$S$) isotopes $f$ rel $\partial S\cup U$ to an embedding $f'$ such that $f'|_D\in\Emb^c(D,X\rel N)$.
  The surjectivity on $\pi_0$ follows from this.
  The same construction for an isotopy gives injectivity on~$\pi_0$.

  Let $\beta^c\colon \Emb^c(D,X\rel N) \to \Emb^c(D,X)$ be the natural forgetful map.
  Since $N$ is a union of collars of the new boundary arcs of $D$, the arguments of \cref{remark:collar-boundary-conditions} show that $\beta^c$ induces a bijection on~$\pi_0$.
  (In fact, it can also be shown that both $\alpha^c$ and $\beta^c$ are weak homotopy equivalences, although we do not use it here.)

  Now we have the following commutative diagram in which horizontal maps are bijections:
  \[
    \begin{tikzcd}
      \pi_0 \Emb^s(S,M \rel U) \ar[d]
      & \pi_0 \Emb^s(D,X\rel N) \ar[l,"\alpha^s_*"',"\approx"] \ar[r,"\beta^s_*","\approx"'] \ar[d]
      & \pi_0 \Emb^s(D,X) \ar[d]
      \\
      \pi_0 \Emb^t(S,M \rel U)
      & \pi_0 \Emb^t(D,X\rel N) \ar[l,"\alpha^t_*","\approx"'] \ar[r,"\beta^t_*"',"\approx"]
      & \pi_0 \Emb^t(D,X)
    \end{tikzcd}
  \]

  The hypothesis that $\partial S$ has a geometric dual implies that $\partial D$ has a geometric dual in~$\partial X$.
  Therefore, the rightmost vertical map $\pi_0\Emb^s(D, X) \to \pi_0\Emb^t(D, X)$ is a bijection, by \cref{theorem:topological=smooth-for-disks}.
  The claim follows from this.

  By Lemma~\ref{lemma:embedding-spaces-bicollar-rel-core},
  $\Emb^s(U,M\rel K)\simeq (\Omega S^2)^{b_1(S)}$, and thus $\pi_0 \Emb^s(U,M\rel K)=0$.

  By Lemma~\ref{lemma:pi1-embedding-bicollar-comparison}, $\pi_1\Emb^s(U,M\rel K) \to \pi_1\Emb^t(U,M\rel K)$ is surjective.

  Using these facts, a five-lemma-type argument yields that \cref{equation:pi0_S_M_relK} is injective.
\end{proof}

Now we state the main result.

\begin{theorem}
  \label{theorem:pi0-surface-comparison}
  Let $M$ be a smooth orientable 4-manifold with nonempty boundary, and let $S\subset M$ be a compact connected smoothly embedded neat surface with nonempty boundary.
  Suppose that $\partial S$ has a geometric dual in~$\partial M$.
  Then $\pi_0\Emb^s(S,M)\to\pi_0\Emb^t(S,M)$ is a bijection.
\end{theorem}

Note that surjectivity follows from~\cref{theorem:light-bulb-smoothing-for-embedding}.
Also, the special case of $S=D^2$ is~\cref{theorem:topological=smooth-for-disks}.

\begin{proof}[Proof of injectivity]
  Consider the homotopy long exact sequences associated with the restriction fibration
  \[
    \Emb^c(S,M \rel K) \to \Emb^c(S,M) \to \Emb^c(K,M)
  \]
  for $c=s$,~$t$.
  These form the following commutative diagram.
  \[
    \begin{tikzcd}[column sep=small]
      \pi_1\Emb^s(K,M) \ar[r] \ar[d]
      & \pi_0\Emb^s(S,M \rel K) \ar[r] \ar[d]
      & \pi_0\Emb^s(S,M) \ar[r] \ar[d,"\Phi"]
      & \pi_0\Emb^s(K,M) \ar[d]
      \\
      \pi_1\Emb^t(K,M) \ar[r]
      & \pi_0\Emb^t(S,M \rel K) \ar[r]
      & \pi_0\Emb^t(S,M) \ar[r]
      & \pi_0\Emb^t(K,M)
    \end{tikzcd}
  \]

  The first (leftmost) vertical map is surjective by Lemma~\ref{lemma:pi1-arcs-comparison}.
  The second vertical map is injective by Lemma~\ref{lemma:pi0-surface-rel-core-injective}.
  The last vertical map is injective by Lemma~\ref{lemma:pi0-arcs-comparison}.
  The preceding lemmas apply with any element of $\Emb^s(S,M)$ as the reference embedding, so the argument applies at every such basepoint.
  Thus, by a five-lemma-type argument, the third vertical map $\Phi$ is injective.
\end{proof}

\begin{proof}[Proof of \cref{theorem:topological-isotopy-implies-smooth-isotopy}]
  It is an immediate consequence of the injectivity in \cref{theorem:pi0-surface-comparison}.
\end{proof}

\paragraph{Unparametrized surfaces in 4-manifolds}

Analogues of \cref{theorem:topological-isotopy-implies-smooth-isotopy,corollary:topological=smooth-for-surfaces} also hold for \emph{unparametrized} embedded surfaces in~$M$.
More precisely, in both the smooth and topological categories, we say that two neat surfaces $S_0$, $S_1 \subset M$ are isotopic if there is an ambient isotopy $\{h_t\colon M \to M \}_{0\le t\le 1}$ with $h_0=\id_M$ and $h_1(S_0)=S_1$. 
We always assume that an isotopy is rel boundary, i.e., $h_t|_{\partial M} = \id_{\partial M}$ for all~$t$.
The following are consequences of~\cref{theorem:topological-isotopy-implies-smooth-isotopy,corollary:topological=smooth-for-surfaces}, obtained using the fact that every homeomorphism of a surface $S$ extending $\id_{\partial S}$ is topologically isotopic rel $\partial S$ to a diffeomorphism.

\begin{corollary}
  \label{corollary:unparametrized-topological-isotopy-implies-smooth-isotopy}
  Let $M$ be a smooth orientable 4-manifold with nonempty boundary.
  For any two smooth neat surfaces in $M$ whose common boundary has a geometric dual in~$\partial M$, topological isotopy implies smooth isotopy.
\end{corollary}

\begin{corollary}
  \label{corollary:unparametrized-topological=smooth-for-surfaces}
  Let $M$ be a smooth orientable 4-manifold with nonempty boundary.
  Fix a smooth link $L\cong S^1\sqcup \cdots \sqcup S^1 \subset \partial M$ that has a geometric dual in~$\partial M$.
  Then the following natural map is a bijection.
  \[
    \left\{
      \begin{array}{@{}c@{}}
        \text{smooth surfaces}\\
        \text{in $M$ bounded by $L$}
      \end{array}
    \right\}
    \Big/
    \begin{array}{@{}c@{}}
      \text{smooth}\\
      \text{isotopy}
    \end{array}
    \;\xrightarrow{\;\approx\;}\;
    \left\{
      \begin{array}{@{}c@{}}
        \text{topological surfaces}\\
        \text{in $M$ bounded by $L$}
      \end{array}
    \right\}
    \Big/
    \begin{array}{@{}c@{}}
      \text{topological}\\
      \text{isotopy}
    \end{array}
  \]
\end{corollary}

\sloppy
\bibliographystyle{amsalpha}
\def\MR#1{}
\bibliography{research}

\end{document}